\documentclass[aps,prl,preprint]{revtex4-1}
\usepackage{amsmath}
\usepackage{amssymb}
\usepackage{amsthm}
\usepackage[utf8]{inputenc}
\usepackage[final]{graphicx}
\usepackage{color}

\newtheorem{definition}{Definition}
\newtheorem{lemma}{Lemma}
\newtheorem{proposition}{Proposition}

\newtheorem{theorem}{Theorem}

\newtheorem{remark}{Remark}
\newtheorem{hypothesis}{Hypothesis}

\begin{document}

\title{A contraction based, singular perturbation approach to near-decomposability in complex systems}
\author{Gabriel D. {Bousquet} and Jean-Jacques E. {Slotine}} 
\affiliation{Nonlinear Systems Laboratory,\\
Massachusetts Institute of Technology, Cambridge, MA 02139, USA}

\begin{abstract}
We revisit the classical concept of near-decomposability in complex systems, introduced by Herbert Simon in his foundational article \emph{The Architecture of Complexity}, by developing an explicit quantitative analysis based on singular perturbations and nonlinear contraction theory. Complex systems are often modular and hierarchic, and a central question is whether the whole system behaves approximately as the ``sum of its parts'', or whether feedbacks between modules modify qualitatively the modules behavior,  and perhaps also generate instabilities. 
We show that, when the individual nonlinear modules are contracting (i.e., forget their initial conditions exponentially), a critical separation of timescales exists between the dynamics of the modules and that of the macro system, below which it behaves approximately as the stable sum of its parts.
Our analysis is fully nonlinear and provides explicit conditions and error bounds, thus both quantifying and qualifying existing results on near-decomposability.

\end{abstract}

\maketitle

The question of modularity and near-decomposability, first popularized by Herbert Simon in his seminal paper \emph{The Architecture of Complexity}~\cite{Simon1962}, has been central to the study of complex systems~\cite{Borner2007,Sporns2015}. 
In Physics, modularity and hierarchy have long been addressed through separation of timescales approaches~\cite{Strogatz2014}, QSSA~\cite{Bowen1963,Segel1989} and perturbation theory. 
With the growing interest in network sciences, in particular systems biology, synthetic biology~\cite{Jayanthi2011,Gyorgy2014} and neurosciences~\cite{Chaudhuri2015,Slotine2001}, where the timescale separation between layers can be quite small, quantitative analysis tools for nonlinear near-decomposable systems are needed. 

Here we revisit the question of near-decomposability from a quantitative viewpoint by  exploiting recent results in singular perturbation theory and nonlinear contraction analysis~\cite{Lohmiller1999,DelVecchio2013}, and provide unifying tools to analyze the robustness and stability of nonlinear near-decomposable systems. 
We will base our discussion on a canonical example loosely inspired from Simon while referring to the more mathematical development of the Appendix where the theorems are rigorously presented.

Consider the dynamic system representing the evolution of temperatures in a perfectly insulated building. The building is composed of $n=2$ floors, each floor has $m=2$ rooms. For the simplicity of exposition, we assume that all rooms have equal heat capacity. The temperature in each room is denoted $x_{ij}$ where $i$ is the floor number and $j$  is the room number. The thermal conductance between rooms on the same floor is much higher than between rooms on separate floors. Overall, the dynamic equation for the evolution of the room temperatures might be
\begin{equation}
  \label{eq:rooms1}
  \begin{aligned}
    \dot x_{11} &= \phantom{-}f(x_{12} - x_{11}) + \epsilon g_1(x_{21} - x_{11})\\
    \dot x_{12} &= -f(x_{12} - x_{11}) + \epsilon g_2(x_{22} - x_{12})\\
    \dot x_{21} &=  \phantom{-}f(x_{22} - x_{21}) - \epsilon g_1(x_{21} - x_{11})\\
    \dot x_{22} &= -f(x_{22} - x_{21}) - \epsilon g_2(x_{22} - x_{12})\\
  \end{aligned}
\end{equation}
where $f,g_1, g_2$ represent possibly nonlinear conductances. $f$ captures the coupling between rooms on the same floor while $g_i$ represent the coupling across floors. When the parameter $\epsilon$ is small, cross-floor coupling is weak and the system is nearly-decomposable.

The intuition suggests that in case of weak cross-floor coupling, on each floor the individual room temperatures, often referred to as microstates or fast variables, should approximately equalize on a fast timescale. In singular perturbation theory, the fast variables are said to converge to the slow manifold.  Later on only the average temperatures, also called macrostates or slow variables, are needed to approximately predict the system's behavior.

It must come at no surprise that stability plays a key role in the present study. The analysis is based on the property of contraction (Definition~\ref{def:contraction}), namely that identical systems started with any two different initial conditions tend towards each other exponentially.  Many complex systems, be they  physical systems, energy networks or information networks, are not stable \emph{stricto sensu} because of the existence of invariants: in our particular example, the total energy $\sum_{ij} x_{ij}$ is conserved and prevents trajectories from converging to each other. While some of our results may still hold for non-contracting systems (Remark~\ref{rq:no_contraction}), it is often useful to reformulate the problem without invariants. In our example, the change of variables $\delta_i = \frac{1}{2}(x_{i2} - x_{i1})$ and $\Delta = \frac{1}{2}\left((x_{22} + x_{21}) - (x_{12} + x_{11})\right)$ both leverages on the invariant to reduce the system's dimensionality  and  transforms system~(\ref{eq:rooms1}) into the standard singular perturbation form~(\ref{eq:ftildegbar}) (\emph{c.f.} Remark~\ref{rq:sp_form}):
\begin{equation}
  \label{eq:rooms2}
  \begin{aligned}
    \dot \delta_{1} &= -f(\delta_1) + \frac{\epsilon}{2} \left(g_2\left({\Delta + \delta_2 - \delta_1}\right) - g_1\left(\Delta + \delta_1 - \delta_2 \right)\right)\\
    \dot \delta_{2} &= -f(\delta_2) + \frac{\epsilon}{2} \left(g_1\left({\Delta + \delta_1 - \delta_2}\right) - g_2\left({\Delta + \delta_2 - \delta_1}\right)\right)\\
    \dot\Delta &= -\epsilon \left(g_1\left({\Delta + \delta_1 - \delta_2}\right) + g_2\left(\Delta + \delta_2 - \delta_1\right)\right)\\
  \end{aligned}
\end{equation}
Note that $\delta_i$ and $\Delta$ are the ``barycentric'' variables for the subsystems and slow system, respectively.

Modular systems such as this one give rise to two questions: is the overall system stable? Does is approximately behave in a modular way, \emph{i.e.} as the sum of its parts? In our example, the modular behavior is as follows: the temperature differences between rooms on the same floor $\delta_i$ converge on a fast timescale to $O(\epsilon)$. On a slow timescale, a good approximation $\bar\Delta$ for the temperature difference between floors evolves following the third line of Equation~(\ref{eq:rooms2}) with $\delta_i=0$. Proposition~\ref{prop:sp_sys} provides criteria guaranteeing stability of the approximate system. The main contribution of this paper, Theorem~\ref{th:sp} provides nonlinear criteria based on contraction and finite-gain L-stability~\cite{Khalil} guaranteeing that the singular perturbation approximation holds while providing explicit error bounds for $\delta_i$ and $\Delta  - \bar\Delta$. Theorem~\ref{th:sp} relies heavily on Lemma~\ref{lem:rob2}. Hypothesis~\ref{hyp:sp} provides an explicit estimate for the minimum timescale separation between the micro and macrostates. This can be used recursively to estimate the timescales in a cascade of modular systems (Remark~\ref{rq:cascade}).

If $f, g_i$ are passive, \emph{i.e.} verify $xf(x)>0$ for all $x\neq 0$, traditional Lyapunov Analysis~\cite{slotine1991applied} is sufficient to prove the system's stability (but not necessarily its approximate dynamics).
Consider however the following situation. Imagine that a small air-conditioning system has been installed between rooms 12 and 22, with the intent to bring the two rooms to equal temperature. Assume now that a mistake has been made and that the thermostats of the two rooms have been switched up. The  reversed air-conditioning system can be seen as a negative conductance and $g_2$ might now take the value $g_2 = -kg_1$ with $k>0$. The building is not passive anymore. For small $\epsilon$ the singular perturbation approximation typically holds but for large $\epsilon$ the approximation breaks down and the system may become unstable. 

As a numerical example, take $f = g_1 = x\mapsto x + \frac{1}{2}\sin x$ and $k = 1/2$. Hypothesis~\ref{hyp:sp} is verified in a constant metric with $\beta_f = 1/2,\beta_g = 1/4,d_f = 3/4, \alpha_{f,x} = \sqrt{2}/4, \alpha_{f,y} = 3/4,d_g = 1, \alpha_{g,x} = \sqrt{2}/2$. Remark~\ref{rq:cascade} indicates $\epsilon_c = \sqrt{2}/7$. In Figures~\ref{fig:0dot5epsc}, \ref{fig:2dot5epsc} and \ref{fig:5epsc}, the time evolution of the system for random initial conditions in $x_{ij} \in [-5,5]$ are plotted for $\epsilon/\epsilon_c = 0.5, 2.5, 5$. For $\epsilon = 0.5\epsilon_c$, the system behaves as expected. For $\epsilon = 2.5\epsilon_c$, there are several equilibria. This is unforeseen by the singular perturbation approximation, in accordance with our results which do not guarantee contraction beyond $\epsilon_c$. For $\epsilon = 5\epsilon_c$, the overall system is often unstable.

The present example is modular in that the fast dynamics $\delta_i$ are weakly coupled with each other: in  Equation~\ref{eq:ftildegbar}, $\tilde f$ is factorizable (Remark~\ref{rq:modular}). It is also important to note (Remarks~\ref{rq:aggregate} and \ref{rq:avg}) that where Simon's describes the influence of the microstates onto the macrostates and across modules as done in an \emph{aggregate way}, it has to be understood in the sense of an aggregation in time $\bar x_{ii}$, not an \emph{a priori} ensemble aggregation of the microstates.

\section{Appendix}
\appendix

\section{General robustness results}
In the following, we assume that all functions are at least continuous and more largely,  have sufficient smoothness. The state $x$ is in $\mathbf{R}^p$ equipped with a norm $|\cdot|$. We define the shorthand $\|\cdot\| = \sup_{\mathbf{R}^p} |\cdot|$. The notation $A\prec B$ (\emph{resp.} $A \preccurlyeq B$) indicates that $B-A$ is positive definite (\emph{resp.} positive semi-definite).
\begin{definition}[Contracting systems]
\label{def:contraction}
  Consider the system $\dot{x} = f(x, t)$. It is said to be
  contracting if all trajectories converge exponentially towards each
  other \cite{Lohmiller1999}. A sufficient condition for contraction is
  that there exist $\beta>0$, called the contraction rate,  and a metric $M(x, t) = \Theta^T\Theta \succ
  0 $  such that
  \[
  F = \dot{\Theta}\,\Theta^{-1} + \Theta\, \nabla \!\!f\, \Theta^{-1}
  \preccurlyeq -\beta I.
  \]
  We assume that the condition number of $\Theta(x, t)$ is bounded and call $\chi$ its maximum. With a slight abuse of language, we will use the term ``metric'' for both $\Theta$ and $M$.

A  system $\dot x = f(x, y, t)$ is partially contracting if it is contracting in $\Theta(x, t)$ for any $y(t)$.
\end{definition}
We will make extensive use of the following robustness property of contracting systems.
\begin{lemma}[Robustness with bounded disturbance]
  Consider the two related dynamic systems
  \begin{equation*}
    \label{eq:rob1}
    \begin{aligned}
      \dot{x} &= f(x, t) + d(x, t)\\
      \dot{x}_0 &= f(x_0, t)\\
    \end{aligned}
  \end{equation*}
with $f$ $\beta, \chi$-contracting and $d$ bounded. Let $R(t) = |x(t) - x_0(t)|$. Then
\[R(t) \le \chi R(0) \exp(-\beta t) + \|d\|\chi/\beta
\]
\end{lemma}

\begin{proof}
$R$ can be expressed in its integral form
\[
R = \left|\int dx \right| = \int |dx|
\]

Using the local transformation $ dx = \Theta^{-1}\delta z$ let $Q(t) = \int|\delta z(t)|$. Using the triangular equality, the following relation is obtained
\begin{equation}
\|\Theta\|^{-1} Q\le R \le \|\Theta^{-1}\|Q.
\label{eq:encadrement}
\end{equation}
The differential dynamics $ d\dot{x} = \nabla f dx + \delta d $ induces $\delta \dot{z} = F\delta z + \Theta\delta d$ and hence
\begin{equation}
 \dot{Q} \le -\beta Q + \|\Theta d\|.
\label{eq:qdot}
\end{equation}
From this, $Q\le Q(0) \exp(-\beta t) + \|\Theta\| \|d\|/\beta$. Utilizing the left-hand side of Equation~(\ref{eq:encadrement}) on $Q(0)$ and the right hand side on $Q$ brings the result.

\end{proof}

\begin{lemma}[Robustness with finite-gain L-stable disturbance]
\label{lem:rob2}
  Consider a system
  \begin{equation}
    \dot{x} = f_0(x, t) + d(x, t)\label{eq:generalizedRobustness}
  \end{equation}
  where $f_0$ is $\beta,\chi$-contracting and $d(x, t)$ is finite-gain L-stable: $|d(x, t)|\le K_0 + K_x|x|$. Assume also that $K_x\le \beta/\chi$ and that the system $\dot{x}_0 = f_0(x_0, t)$ has a forward bounded solution $\|x_0(t)\|\le x_{00}$.
Then
\[
  |x(t) - x_0| \le |x(0) - x_0(0)|\chi e^{-\left(\beta - \chi K_x\right)t} + \chi\frac{K_1 + K_xx_{00}}{\beta - \chi K_x}
  \]
\end{lemma}

\begin{proof}
  The triangular inequality is applied to $ x = (x - x_0) + x_0$ to get $|x| \le R + |x_0|$. With the bound on $d$ we obtain
\[
  |\Theta d| \le |\Theta|K_xR + |\Theta| (K_1 + K_xx_{00})
\]
The proof is completed by reinjecting this relation into equation~\ref{eq:qdot} and using the right hand site of equation~\ref{eq:encadrement} to dominate $R$
\[
 \dot{Q} + (\beta - K_x\chi) Q \le |\Theta| (K_1 + K_xx_{00})
\]
in a similar fashion to the previous lemma.
\end{proof}

\section{Robustness of singularly perturbed systems}

\begin{proposition}
\label{prop:sp_sys}
  Assume that Hypothesis~\ref{hyp:ftilde} below holds. Then the system
  \begin{equation}
    \begin{aligned}
      \dot{\tilde x} &= \tilde f(\tilde x,  y, u(\epsilon t))\\
      \dot{ y} &= \epsilon \bar g( y, u(\epsilon t))
    \end{aligned}
\label{eq:spfull}
\end{equation}
is contracting.
\end{proposition}

\begin{hypothesis}
The $ y$ dynamics is contracting with metric $\Theta_y$.
 The system $ \tilde f(\tilde x,  y, u(\epsilon t))$ is partially contracting in $\tilde x$ in the metric $\Theta_x$. $B = \frac{1}{2} \Theta_x \nabla_y\tilde f \Theta_y^{-1}$ is uniformly bounded. The symmetric part of the generalized Jacobian for the $ y$ dynamics $F_y^s$ is bounded below.
\label{hyp:ftilde}
\end{hypothesis}

\begin{proof}
The system is hierarchic~\cite{Slotine2003}. With the state $(\tilde x,  y)$ it is contracting in the metric  $\Theta = \text{diag} (\epsilon^{1/2}\nu\Theta_x, \Theta_y)$ for sufficiently small $\nu$. Indeed, the symmetric part of the  generalized Jacobian is
\begin{equation}
  \begin{pmatrix}
    F_x^s & \epsilon^{1/2}\nu B\\
     * & \epsilon F_y^s\\
  \end{pmatrix}
\end{equation}
Using Schurr's complement, it is negative definite iff $F_x^s - \nu^2 B^T\left(F_y^s\right)^{-1}B\prec 0$, which holds for $\nu$ sufficiently small, since by hypothesis $B$ and $\left(F_y^s\right)^{-1}$ are bounded.
\end{proof}

\begin{lemma}
  Consider the two systems:
  \begin{equation}
    \begin{aligned}
      \dot{\tilde x} &= \tilde f(\tilde x, y, u(\epsilon t)) + \epsilon \delta f(\tilde x, y, u(\epsilon t))\\
      \dot{y} &= \epsilon \bar g (y, u(\epsilon t)) + \epsilon \delta g(\tilde x, y, u(\epsilon t))
    \end{aligned}
    \label{eq:ftildegbar}
  \end{equation}
and
\begin{equation}
  \begin{aligned}
    \dot{\bar y} &= \epsilon \bar g(\bar y, u(\epsilon t))
    \label{eq:ybardot}
  \end{aligned}
\end{equation}
with initial condition $(O(\epsilon), y_0 + \Delta)$ and $y_0$ respectively.
Assume that Hypothesis~\ref{hyp:sp} stated below holds.

Then, for all $t$, $\tilde x = O(\epsilon)$. After decay of the transient $e^{-\epsilon\beta_g t}$,   $y - \bar y =  \tilde y = O(\epsilon)$.
\label{lem:Oepscvg}
\end{lemma}
\begin{hypothesis}
\label{hyp:sp}
$\tilde f$ is partially $(\chi_f, \beta_f)$-contracting with respect to $\tilde x$ and $\tilde x = 0$ is an equilibrium for all $y, u(\epsilon t)$. $\bar g(y, u(\epsilon t))$ is $(\chi_g, \beta_g)$-contracting. The solution of system~(\ref{eq:ybardot}) is bounded $|\bar y(t)|\le \bar M$. $u$ is bounded. The perturbations are bounded in the following way:
  $|\delta f| \le d_f + \alpha_{f,x} |\tilde x| + \alpha_{f, y} |y|$
and
  $|\delta g| \le d_g+  \alpha_{g,x} |\tilde x| $.
Furthermore, $\epsilon\frac{\chi_f}{\beta_f}\left(\alpha_{f,x} +\frac{\chi_g}{\beta_g}\alpha_{f,y}\alpha_{g,x}\right)<1$.
\end{hypothesis}

\begin{proof}
  There  is a maximal solution on $t \in \mathbf{R}^+$. Let $M_{\tilde x}(t)= \max_{\tau\le t} |\tilde x(\tau)|$ and similarly define $M_y(t)$ and $M_{\tilde y}(t)$. $0$ is solution of the unperturbed $\dot{\tilde x} = \tilde f(\tilde x, y, u(\epsilon t))$. The robustness property applied to the two lines of Equation~(\ref{eq:ftildegbar}) gives on $[0, t]$

  \begin{equation}
    \begin{aligned}
      |\tilde x(t)| &\le \chi_f|\tilde x_0| + \frac{\epsilon\chi_f}{\beta_f}(d_f + \alpha_{f,x} M_{\tilde x}(t) + \alpha_{f, y} M_y(t)\\
      |\tilde y(t)| &\le \chi_g|\Delta| + \frac{\chi_g}{\beta_g} \left( d_g + \alpha_{g,x} M_{\tilde x}(t)\right)\\
    \end{aligned}
\end{equation}
Note that $M_y(t)\le \bar M + M_{\tilde y}(t)$. Taking the supremum of the expression and reorganizing we get after some algebra
\begin{equation}
  \begin{aligned}
    M_{\tilde x}(t)&\le 
\frac{
  \chi_f |\tilde x_0| + \frac{\epsilon\chi_f}{\beta_f}
  \left(d_f + \alpha_{f,y}(\bar M + \chi_g \left(|\Delta| + d_g/\beta_g\right)\right)
}{
  1 - \epsilon\frac{\chi_f}{\beta_f}\left(\alpha_{f,x} +\frac{\chi_g}{\beta_g}\alpha_{f,y}\alpha_{g,x}  \right)
}
  \end{aligned}
\end{equation}
$M_{\tilde x}(t)$ is increasing and bounded, therefore it converges to a finite limit $M_{\tilde x}$. The boundedness of $\tilde x$ induces that of $\tilde y$ and $y$. $\tilde x_0 = O(\epsilon)$ implies $M_{\tilde x} = O(\epsilon)$.
Finally, the robustness property with bounded disturbances can be applied to the $y$ dynamics:
$|\tilde y|\le \chi_g|\Delta|e^{-\epsilon\beta_g t} + \frac{\chi_g\alpha_{g,x} M_{\tilde x}}{\beta_g}$ which becomes $O(\epsilon)$ after the transient.
\end{proof}

\begin{lemma}
\label{lem:sp_fast}
  Consider the system~(\ref{eq:ftildegbar}) with initial conditions $(\tilde x_0, y_0)$. Assume that Hypothesis~\ref{hyp:sp} holds. Then at $t = \log(1/\epsilon)/\beta_f$, the system has evolved to $(O(\epsilon), y_0 + O(\epsilon t))$.
\end{lemma}

\begin{proof}
  By continuity of $\bar g$ and boundedness of $u$, there is $d_{y_0}, \alpha_{y_0}>0$ and a neighborhood $B_{y_0} = \{y: |y - y_0|\le c_{y_0}\}$ such that  $|\bar g(y, u)|\le  d_{y_0} + \alpha_{y_0}|y - y_0|$ (take for instance $d_{y_0} = \max_u \bar g(y_0, u)$ and $\alpha_{y_0} = 2 \max_u |\frac{\partial \bar g}{\partial y}\scriptstyle{(y_0, u)}|$).

The robustness relation for $\tilde x$ and integration of the $y$ dynamics gives:
  \begin{equation}
    \begin{aligned}
      |\tilde x(t)| &\le \chi_f|\tilde x_0| + \frac{\epsilon\chi_f}{\beta_f}(d_f + \alpha_{f,x} M_{\tilde x}(t) + \alpha_{f, y} M_y(t))\\
     |y(t) - y_0| &\le  \epsilon\int_0^t|\bar g(y(\tau), u)| + \alpha_{g,x} M_{\tilde x}(\tau)d\tau\\
    \end{aligned}
\end{equation}

By continuity, $y$ remains in the neighborhood $B_{y_0}$ at least for some time. Assume that $y$ does leaves the neighborhood at some point and let $t_c$ be the first time when the neighborhood boundary is crossed. Then, 
\[
\begin{aligned}
  c_{y_0} &= \left|\int_0^{t_c}\dot{y}d\tau\right|\\
  &\le \epsilon t_c\bigg( d_{y_0} + \alpha_{y_0}c_{y_0}+ \\
  &\phantom{ \epsilon t_c\bigg( d_{y_0}} +\alpha_{g,x}g\frac{\chi_f |\tilde x_0| + \epsilon\chi_f/\beta_f(d_f + \alpha_{f,y}(|y_0| + c_{y_0})}{1 - \epsilon\chi_f\alpha_{f,x}/\beta_f}     \bigg)\\
&= \epsilon t_c(K_1 + O(\epsilon))\\
\end{aligned}
\]
where $K_1$ is a constant gathering the terms of the second line. Below some critical  $\epsilon_1$, in the last expression the $O(\epsilon)$ can be bounded by some other constant $K_2 = \limsup_{\epsilon\to 0} O(\epsilon)$.
As a consequence, $y$ remains in $B_{y_0}$ at least until $t_c\ge t_0/\epsilon$ with $t_0 = c_{y_0}/K$.

We can now bound 
\[
|y(t) - y_0|\le \epsilon t K, \quad t\le t_0/\epsilon
\]
If $y$ never leaves the $B_{y_0}$ neighborhood, the relation also holds. As a consequence, $y$ can be bounded by a constant quantity for a time arbitrary large $t\sim 1/\epsilon$: $M_y(t_0/\epsilon) \le |y_0| + t_0K$. Such a bound still holds if $y$ never leaves $B_0$. We can now apply once more the robustness property to the $\tilde x$ dynamics in order to obtain the tighter bound
\[
|\tilde x|\le \chi_f |\tilde x_0|e^{-\beta_f t} + O(\epsilon),\quad t\le t_0/\epsilon
\]
At $t = \log(1/\epsilon)/{\beta_f}$, the first terms becomes $O(\epsilon)$. Which proves the result. Note that for $\epsilon$ small enough, this time is smaller than $t_0/\epsilon$.

\end{proof}

\begin{theorem}
\label{th:sp}
  Consider the systems (\ref{eq:ftildegbar}) and (\ref{eq:ybardot}). Assume that Hypothesis \ref{hyp:sp} holds. After a transient $\left(\frac{1}{\beta_f} + \frac{1}{\epsilon\beta_g}\right)\log\frac{1}{\epsilon}$, system~(\ref{eq:ftildegbar}) remains in a $O(\epsilon)$ neighborhood of $(0, \bar y)$. An explicit bound for the neighborhood is shown in the proof of Lemma~\ref{lem:Oepscvg}.
\end{theorem}

\begin{proof}
  This is the direct consequence of subsequently applying the Lemmas~\ref{lem:Oepscvg} and \ref{lem:sp_fast}.
\end{proof}

\begin{remark}
\label{rq:no_contraction}
 The contraction property of $\bar g$ is utilized to obtain the bound in $\tilde y$ in Lemma~\ref{lem:Oepscvg}. It is sufficient to assume that $y$ is bounded, perhaps using general passivity properties of the full system, to obtain the bound on $\tilde x$. Note that the boundedness of $y$ is guaranteed if $\dot y = \bar g(y, u)$ is contracting and has \emph{a} bounded solution.
\end{remark}
\begin{remark}
  It is straightforward to obtain explicit bounds in place of the $O(\epsilon)$.
\end{remark}
\begin{remark}
  Hypothesis~\ref{hyp:sp} is not particularly restrictive. In particular it holds for linear systems and finite-gain L-stable systems. In contrast, the hypotheses for  Lemma~1 of \cite{DelVecchio2013} require in essence $\delta f$ to be bounded, while we only require it to be finite-gain L-stable in $\tilde x$.
\end{remark}

\begin{remark}
\label{rq:sp_form}
  In general, a transformation is required to make the condition $\tilde{f}(0, y, u) = 0$ in Hypothesis~\ref{hyp:sp} hold. Consider a system in the form
\begin{equation}
  \begin{aligned}
   \dot x &= f(x, y, u(\epsilon t)) + \epsilon g_{x}(x, y, t)\\
   \dot y &= \epsilon g_{y}(x, y, t)
  \end{aligned}
\label{eq:sys}
\end{equation}
with $f(0, y, u)$ not necessarily 0. Assume however that $f$ is partially contracting in $x$ in  metric $\Theta_{x}$ with rate and condition numbers $\beta_x, \chi_x$.

Assume that $f(x, y, u) = 0$ has a solution $\bar x(y, u)$. The partial contraction of $f$ ensures that it is unique. It represents the equilibrium value of $x$ after suppressing the forcing $\epsilon g_x$ and freezing $y, u$.
After the change of variable $\tilde x = x - \bar x$, system~(\ref{eq:sys}) is transformed into system~(\ref{eq:ftildegbar}) with 
\begin{equation}
\begin{aligned}
  \tilde f(\tilde x, y, u) &= f(\bar x(y, u) + \tilde x, y, u)\\
  \delta f &=  -\frac{\partial\bar x}{\partial y} (\bar g_x(\tilde x, y, u) + \epsilon \delta g) - \frac{\partial\bar x}{\partial u}u'\\
  \bar g(y, u) &= g(y, \bar x(y, u), u)\\
  \delta g &=  g_y(\tilde x + \bar x, y, u, t) - \bar g(y, u, t)\\
\end{aligned}\label{eq:ftoftilde}
\end{equation}

\end{remark}

\begin{remark}
\label{rq:modular}
Modular systems may take the form $x = (x_1,\dots,x_n)$ such that the first line of Equation~(\ref{eq:ftildegbar}) can be factored into
\begin{equation}
   \dot x_i = f_{x_i}(x_i, y, u_i) + \epsilon g_{x_i}(x, y, u)
\end{equation}
The key here is that for each subsystem $i$, $\tilde f_i$ depends only on the local microstates $x_i$ and can be analyzed separately: $\tilde f$ is partially contracting in the metric $\Theta_x = \mathrm{diag}(\Theta_{x_1}, \dots, \Theta_{x_n})$ with $\beta_x = \min\beta_{x_i}$ and $\chi_x = \max\chi_{x_i}$.
\end{remark}

\begin{remark}
  \label{rq:aggregate}
  The influence of the individual $x_i$ \emph{a priori} propagates into the macrostates and other modules through their quasi-steady values $\bar x_i$ (see lines 1 and 3 of Equation (\ref{eq:ftoftilde})). It might be that the $x_i$ only have a group influence if the coupling is not only weak, as is assumed in this work, but also low rank (equivalently of higher order). The latter property is about the  graph connectivity~\cite{Newmann2003}, and is quite different from the property of connection strength discussed here.
\end{remark}

\begin{remark}
\label{rq:avg}
  It can be shown that the important aspect of the micro/macrostate separation is that the microstates converge \emph{on average} to $\bar x_i$. The qualitative behavior still holds if $x_i \to \bar x_i + \delta x_i$ with the time average $\overline{\delta x_i}$ small on a fast timescale. In other words our analysis in the singular perturbation framework extends to the averaging framework.
\end{remark}

\begin{remark}
  \label{rq:cascade}

In Hypothesis~\ref{hyp:sp} it is required that  $\epsilon<\frac{\chi_f}{\beta_f}\left(\alpha_{f,x} +\frac{\chi_g}{\beta_g}\alpha_{f,y}\alpha_{g,x}\right) = \epsilon_c$. As the proof shows, this property is paramount in  guaranteeing the boundedness of $\tilde x$. This constraint is related to small-gain stability. It gives a quantitative estimate of the timescale separation necessary for the the micro/macrostate separation to hold.

If a modular system is constituted of a cascade of modules, we expect the separation between the timescale  $\tau_\mu$ of the smallest system and that of the largest system $\tau_M$, to be $\tau_\mu = (\epsilon_{c_\mu}\dots\epsilon_{c_M})\tau_M$. The estimate  holds even for nonlinear systems.
\end{remark}

\bibliography{library.bib}

\newpage
\begin{figure}
  \centering
  \includegraphics[width=\textwidth]{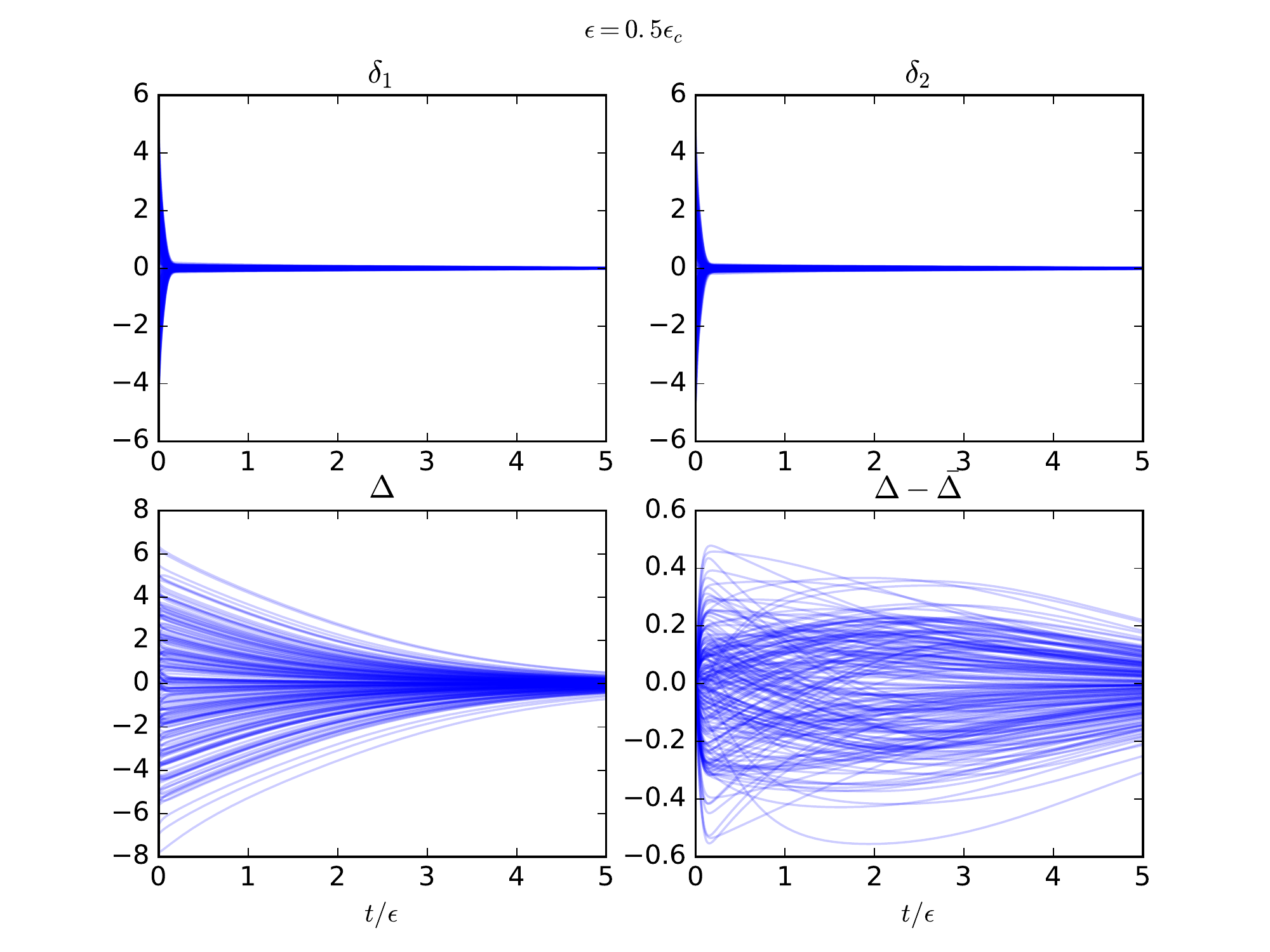}
  \caption{}
  \label{fig:0dot5epsc}
\end{figure}

\begin{figure}
  \centering
  \includegraphics[width=\textwidth]{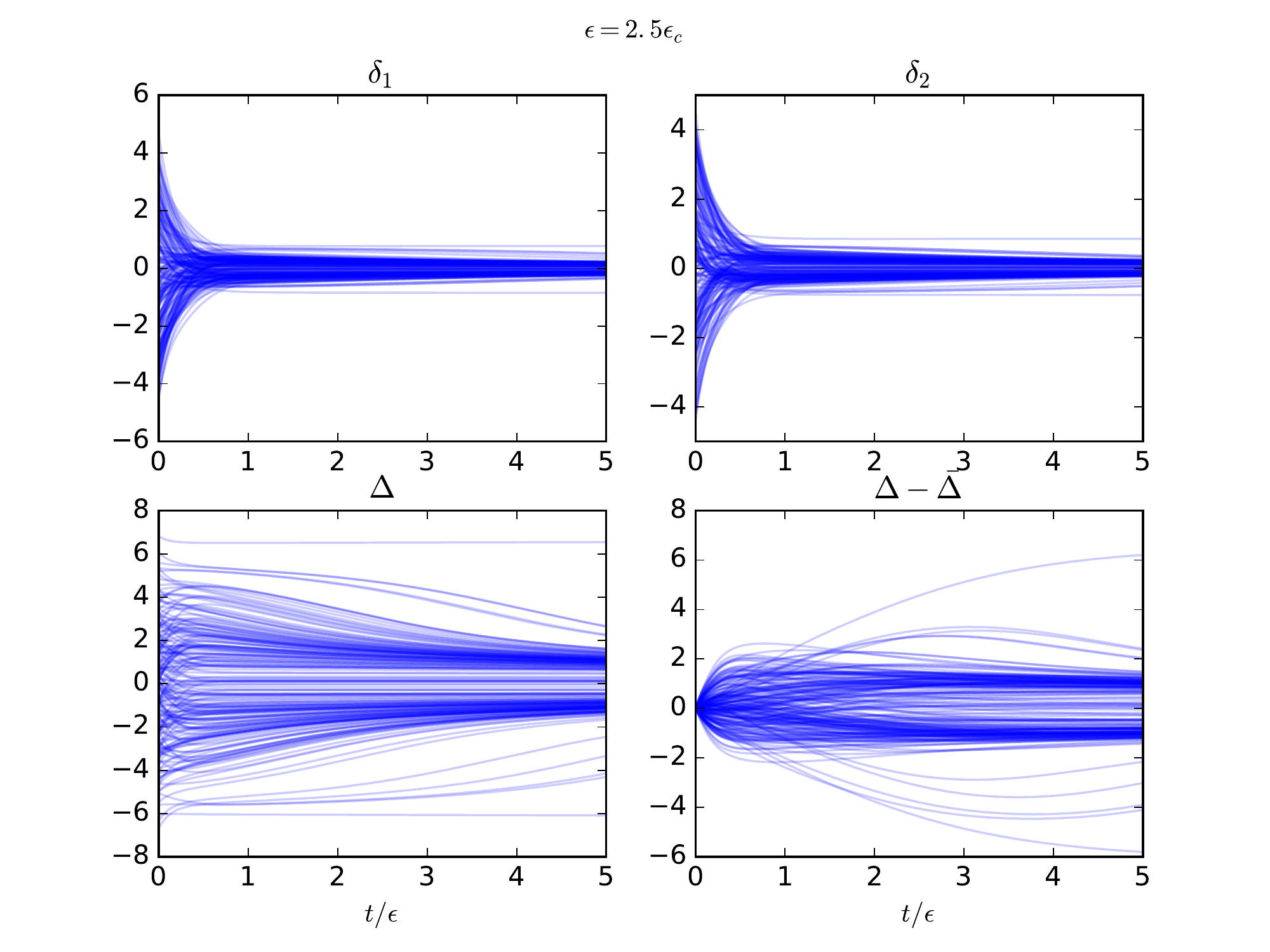}
  \caption{}
  \label{fig:2dot5epsc}
\end{figure}

\begin{figure}
  \centering
  \includegraphics[width=\textwidth]{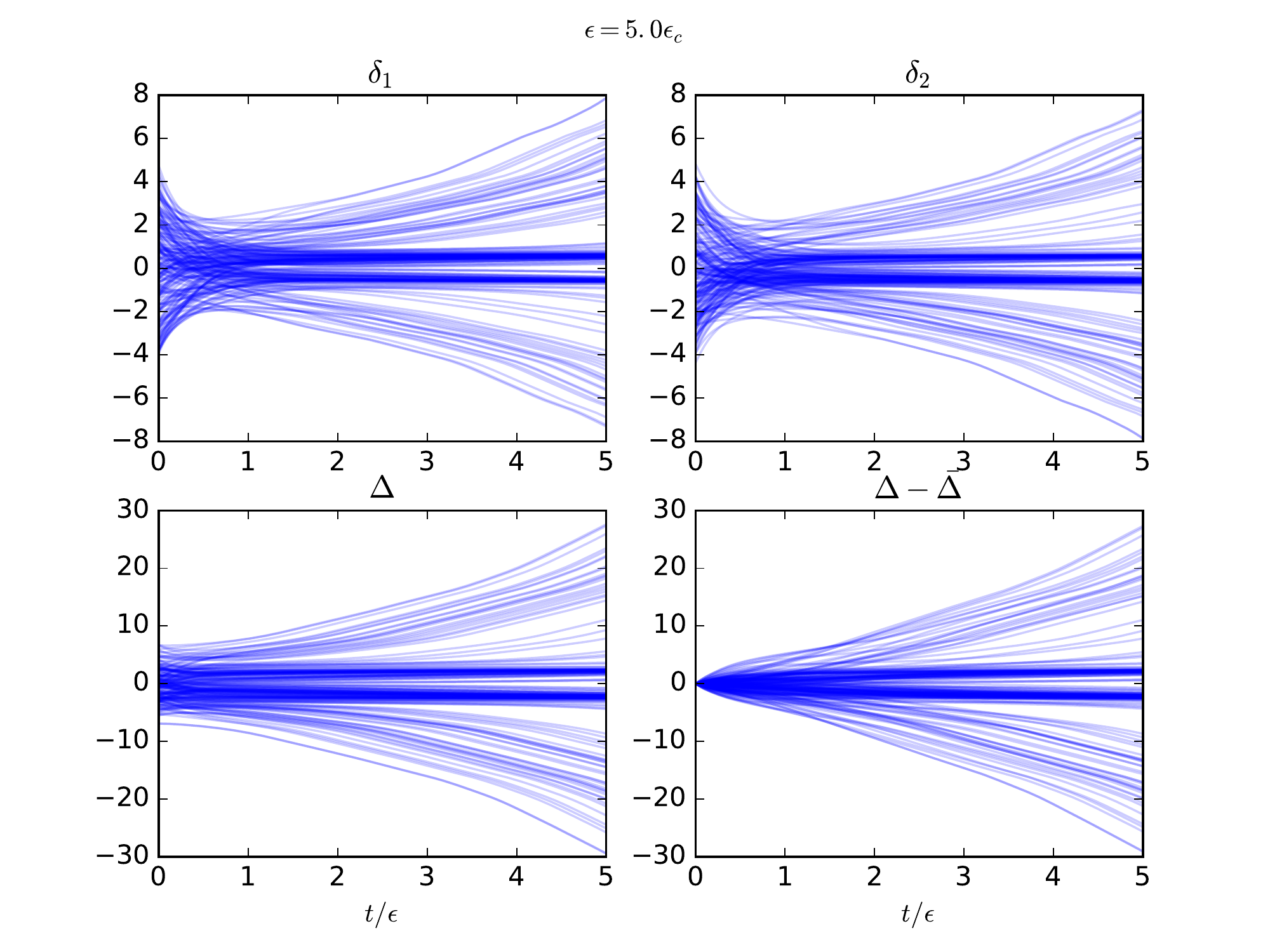}
  \caption{}
  \label{fig:5epsc}
\end{figure}

\end{document}